\documentclass[preprint]{elsarticle}

\usepackage{amssymb,amsmath,amsthm,latexsym}
\usepackage{times}
\usepackage{psfrag}
\usepackage{graphicx}
\usepackage{url}
\usepackage{psfrag}
\usepackage{graphpap}
\usepackage{pstricks}
\usepackage{pst-node}
\usepackage{relsize}

\usepackage[left=2cm,right=2cm,top=3cm,bottom=3cm,a4paper]{geometry}

\theoremstyle{plain}
\newtheorem{Thm}{Theorem}[section]
\newtheorem{Lem}[Thm]{Lemma}

\newtheorem{Cor}[Thm]{Corollary}

\newcommand{\bZ}{\ensuremath{\mathbb{Z}}}

\newcommand{\cF}{\ensuremath{\mathcal{F}}}

\journal{arXiv}

\begin{document}
\begin{frontmatter}

\title{The competition numbers of ternary Hamming graphs}

\author[label1]{{\sc Boram PARK} \corref{cor1}\fnref{label3}}
\author[label2]{{\sc Yoshio SANO} \fnref{label4} }

\address[label1]{Department of Mathematics Education,
Seoul National University, Seoul 151-742, Korea}
\address[label2]{Pohang Mathematics Institute, POSTECH, Pohang 790-784, Korea}
\fntext[label3]{
{\it E-mail addresses:}
{\tt kawa22@snu.ac.kr}; {\tt borampark22@gmail.com} \\
This research was supported by Basic Science Research Program
through the National Research Foundation of Korea (NRF) funded by
the Ministry of Education, Science and Technology (700-20100058).}
\fntext[label4]{
{\it E-mail addresses:}
{\tt ysano@postech.ac.kr}; {\tt y.sano.math@gmail.com} \\
This work was supported by Priority Research Centers Program
through the National Research Foundation of Korea (NRF) funded by
the Ministry of Education, Science and Technology (MEST) (No. 2010-0029638).}
\cortext[cor1]{Corresponding author}

\begin{abstract}
It is known to be a hard problem to compute the competition number $k(G)$
of a graph $G$ in general.
Park and Sano \cite{PS1} gave the exact values
of the competition numbers of Hamming graphs $H(n,q)$
if $1 \leq n \leq 3$ or $1 \leq q \leq 2$.
In this paper, we give an explicit formula of the competition numbers
of ternary Hamming graphs.
\end{abstract}


\end{frontmatter}

\section{Introduction}

The notion of a competition graph was introduced by Cohen \cite{Cohen1}
as a means of determining the smallest dimension of ecological phase space.
The {\it competition graph} $C(D)$ of a digraph $D$ is
a (simple undirected) graph
which has the same vertex set
as $D$ and an edge between vertices $u$ and $v$ if and only if
there is a vertex $x$ in $D$ such that $(u,x)$ and $(v,x)$ are arcs of $D$.
For any graph $G$, $G$ together
with sufficiently many isolated vertices is the competition graph of an
acyclic digraph.
Roberts \cite{MR0504054} defined the {\em competition number} $k(G)$ of
a graph $G$ to be the smallest number $k$ such that $G$ together with
$k$ isolated vertices is the competition graph of an acyclic
digraph.
Opsut \cite{MR679638} showed that the computation of the
competition number of a graph is an NP-hard problem
(see \cite{kimsu} for graphs whose competition numbers are known).
It has been one of important research problems in the study of
competition graphs to compute the competition numbers for
various graph classes
(see \cite{KLS}, \cite{KPS-Johnson}, \cite{KimSano},
\cite{LC}, \cite{LiChang}, \cite{PLK},
\cite{PKS1}, \cite{PKS2}, \cite{Sano}, \cite{ZhaoChang2009},
\cite{ZhaoChang2010}, for recent research).
For some special graph families, we have explicit formulas
for computing competition numbers.
For example, if $G$ is a chordal graph without isolated vertices
then $k(G)=1$,
and if $G$ is a nontrivial triangle-free connected graph
then $k(G)=|E(G)|-|V(G)|+2$ (see \cite{MR0504054}).

In this paper, we study the competition numbers of Hamming graphs.
For a positive integer $q$,
we denote a $q$-set $\{1,2 \ldots, q \}$ by $[q]$.
Also we denote the set of $n$-tuple over $[q]$ by $[q]^n$.
For positive integers $n$ and $q$,
the {\it Hamming graph} $H(n,q)$
is the graph which has the vertex set ${[q]^n}$ and
in which
two vertices $x$ and $y$ are adjacent if $d_H(x,y)=1$,
where $d_H:[q]^n \times [q]^n \to \bZ$ is
the {\it Hamming distance} defined by
$d_H(x,y):=|\{i \in [n] \mid x_i \neq y_i \}|$.
Note that the diameter of the Hamming graph
$H(n,q)$ is equal to $n$ if $q \geq 2$
and that the number of edges of
the Hamming graph $H(n,q)$ is equal to $\frac{1}{2} n(q-1)q^n$.
Park and Sano \cite{PS1} gave the exact values
of the competition numbers of Hamming graphs $H(n,q)$
if $1 \leq n \leq 3$ or $1 \leq q \leq 2$ as follows.
For $n \geq 1$, $H(n,1)$ has no edge and so $k(H(n,1))=0$.
For $n \geq 1$, $H(n,2)$ is a $n$-cube and $k(H(n,2))=(n-2)2^{n-1} +2$.
For $q \geq 2$, $H(1,q)$ is complete graph and so $k(H(1,q))=1$.
\begin{Thm}[\cite{PS1}] \label{thm:H(2,q)}
For $q \geq 2$, we have $k(H(2,q)) =2$.
\end{Thm}
\begin{Thm}[\cite{PS1}]\label{thm:H(3,q)}
For $q \geq 3$, we have $k(H(3,q)) =6$.
\end{Thm}

In this paper, we give the exact values of
the competition numbers of ternary Hamming graphs $H(n,3)$.
Our main result is the following:

\begin{Thm}\label{Mainthm:H(n,3)}
For $n \geq 3$, we have
\[
k(H(n,3)) =(n-3)3^{n-1}+6.
\]
\end{Thm}

\section{Preliminaries}

We use the following notation and terminology in this paper.
For a digraph $D$,
a sequence $v_1, v_2, \ldots, v_n$ of the vertices of $D$
is called an \emph{acyclic ordering} of $D$
if $(v_i,v_j) \in A(D)$ implies $i<j$.
It is well-known that a digraph $D$ is acyclic
if and only if there exists an acyclic ordering of $D$.
For a digraph $D$ and a vertex $v$ of $D$,
we define
the {\it in-neighborhood} $N^{-}_D(v)$
of $v$ in $D$ to be the set $\{w \in V(D) \mid (w,v)\in A(D)\}$,
and a vertex in $N^{-}_D(v)$ is called an {\it in-neighbor} of $v$ in $D$.
For a graph $G$ and a nonnegative integer $k$,
we denote by $G\cup I_k$ the graph such that
$V(G\cup I_k)=V(G)\cup I_k$ and
$E(G\cup I_k)=E(G)$, where $I_k$ is a set of $k$ isolated vertices.

For a clique $S$ of a graph $G$
and an edge $e$ of $G$,
we say {\it $e$ is covered by $S$}
if both of the endpoints of $e$ are contained in $S$.
An \textit{edge clique cover} of a graph $G$
is a family of cliques of $G$ such that
each edge of $G$ is covered by some clique in the family.
The {\it edge clique cover number} $\theta_E(G)$ of a graph $G$
is the minimum size of an edge clique cover of $G$.
An edge clique cover of $G$
is called a {\it minimum edge clique cover} of $G$
if its size is equal to $\theta_E(G)$.

Let $\pi_j:[q]^n \to [q]^{n-1}$ be a map defined by
$
(x_1, ..., x_{j-1}, x_j, x_{j+1}, ..., x_n) \mapsto
(x_1, ..., x_{j-1}, x_{j+1}, ..., x_n).
$
For $j \in [n]$ and ${ p} \in [q]^{n-1}$,
we let
\begin{equation}\label{eq:Sjp}
S_j(p) := \pi^{-1}_j(p) = \{x \in [q]^n \mid \pi_j(x)= { p} \}.
\end{equation}
Note that $S_j(p)$ is a clique of $H(n,q)$ with size $q$. Let
\begin{equation}\label{eq:Fnq}
\cF(n,q):=\{S_j(p) \mid j \in [n], p \in [q]^{n-1} \}.
\end{equation}
Then $\cF(n,q)$ is the family of maximal cliques of $H(n,q)$.
Park and Sano \cite{PS1} showed the following:

\begin{Lem}[\cite{PS1}]\label{lem:LB2}
The following hold:
\begin{itemize}
\item{}
Let $n\ge 2$ and $q \ge 2$,
and let $K$ be a clique of $H(n,q)$ with size at least $2$.
Then there exists a unique maximal clique $S$ of $H(n,q)$
containing $K$.
\item{}
The edge clique cover number of $H(n,q)$ is equal to $nq^{n-1}$.
\item{}
Any minimum edge clique cover of $H(n,q)$ consists of
edge disjoint maximum cliques.
\item{}
The family $\cF(n,q)$ defined by (\ref{eq:Fnq})
is a minimum edge clique cover of $H(n,q)$.
\end{itemize}
\end{Lem}

Now we present the following lemma:

\begin{Lem} \label{lem:ExistD_nq}
Let $G$ be a graph. Suppose that any edge of $G$
is contained in exactly one maximal clique of $G$.
Let $\cF$ be the family of all maximal cliques of size
at least two in $G$,
and let $k$ be an integer with $k\ge k(G)$.
Then there exists an acyclic digraph $D$
satisfying the following:
\begin{itemize}
\item[{\rm (a)}]
The competition graph of $D$ is $G \cup I_{k}$.
\item[{\rm (b)}]
For any vertex $v$ in $D$,
$N^-_D(v) \in \cF \cup \{\emptyset\}$.
\item[{\rm (c)}]
The number of vertices which have no in-neighbor in $D$ is
equal to
$k + |V(G)| - |\cF|$.
\end{itemize}
\end{Lem}

\begin{proof}
Let $k$ be an integer such that $k \ge k(G)$.
By the definition of the competition number of a graph,
there exists an acyclic digraph satisfying (a).
Suppose that any acyclic digraphs satisfying (a) does not satisfies (b).
Let $D$ be an acyclic digraph which maximizes
$|\{ v \in V(D) \mid N^-_D(v)=\emptyset\}|$
among all acyclic digraphs satisfying (a) but not (b).
Since $D$ does not satisfy (b), there exists a vertex $v^*$ in $D$
such that $N^-_D(v^*) \not\in \cF \cup \{ \emptyset \}$.
Since
$D$ maximizes $|\{ v \in V(D) \mid N^-_D(v)=\emptyset\}|$,
we may assume that
$|N^-_D(v)| \neq 1$ for any $v \in V(D)$.
By the assumption
that any edge of $G$
is contained in exactly one maximal clique of $G$,
any clique of size at least two is contained in
a unique maximal clique in $G$.
Since $N^-_D(v^*)$ is a clique of size at least two,
there exists a unique maximal clique $S \in \cF$ containing $N^-_D(v^*)$.

Let $\sigma:=(v_1,v_2,\ldots,v_{|V(D)|})$ be an acyclic ordering of $D$,
i.e., if $(v_i,v_j) \in A(D)$ then $i<j$.
Let $v_i$ and $v_j$ be two vertices of $S$
whose indices are largest in the acyclic ordering $\sigma$.
Since $v_i$ and $v_j$ are adjacent in $G$,
there is a vertex $v_l \in V(D)$ such that
$(v_i,v_l)$ and $(v_j,v_l)$ are arcs of $D$.
Then we define a digraph $D_1$ by $V(D_1) = V(D)$ and
\begin{eqnarray*}
A(D_1) &=& (A(D) \setminus \{ (x, v^*) \mid x \in N^-_D(v^*) \} )
\cup  \{ ( x, v_l) \mid x \in S \}.
\end{eqnarray*}
Then the digraph $D_1$ is acyclic,
since the index $l$ of $v_l$ is larger
than the index of any vertex in $S$.
By the assumption
that any edge of $G$
is contained in exactly one maximal clique of $G$,
we have $N^-_D(v_l) \subseteq S$, which implies that $C(D_1) = G \cup I_k$.
Since
$|\{ v \mid N^-_{D_1}(v) = \emptyset \}|
> |\{ v \mid N^-_{D}(v) = \emptyset \}|$,
we reach a contradiction to the choice of $D$.
Thus, there exists an acyclic digraph satisfying both (a) and (b).

Let $D_2$ be an acyclic digraph satisfying both (a) and (b).
If the digraph $D_2$ has two vertices $u,w$
such that $N^-_{D_2}(u)=N^-_{D_2}(w) \neq \emptyset$,
then we can obtain an acyclic digraph $D_3$
such that all the nonempty in-neighborhoods of vertices in $D_3$
are distinct by deleting arcs in
$\{(x,u) \mid x \in N^-_{D_2}(u) \}$ from $D_2$.
Therefore,
we may assume that all the
nonempty in-neighborhoods $N^-_{D_3}(v)$ are distinct.
Then the number of vertices $v$ such that
$N^-_{D_3}(v) \neq \emptyset$ is exactly equal to $|\cF|$.
Therefore,
$
|\{ v \in V(D_3) \mid N^-_{D_3}(v) = \emptyset\}|
=|V(D_3)|- |\cF|= k+|V(G)|-|\cF|.
$
Thus the digraph $D_3$ satisfies (a), (b), and (c).
Hence the lemma holds.
\end{proof}

By Lemmas \ref{lem:LB2} and \ref{lem:ExistD_nq},
the following holds.

\begin{Cor}\label{cor:Dabc}
Let $k$ be an integer with $k\ge k(H(n,q))$.
Then there exists an acyclic digraph $D$
satisfying the following:
\begin{itemize}
\item[{\rm (a)}]
$C(D)=H(n,q) \cup I_{k}$.
\item[{\rm (b)}]
$N^-_D(v) \in \cF(n,q) \cup \{\emptyset\}$
for any $v \in V(D)$.
\item[{\rm (c)}]
$|\{ v \in V(D) \mid N^-_D(v) = \emptyset\}| = k - (n-q)q^{n-1}$.
\end{itemize}
\end{Cor}

\section{Proof of Theorem \ref{Mainthm:H(n,3)}}

In this section, we present the proof of the main result.
The following theorem gives an upper bound for
the competition numbers of Hamming graphs $H(n,q)$ where $2 \leq q \leq n$.

\begin{Thm} \label{thm:UH_nq}
For $2 \leq q \leq n$, we have
\[
k(H(n,q)) \leq (n-q)q^{n-1} + k(H(q,q)).
\]
\end{Thm}

\begin{proof}
We prove the theorem by induction on $n$.
If $n=2$, then $n=q=2$ and the theorem trivially holds.
For simplicity, for any $2 \leq q \leq n$,
let
$
\alpha(n,q) := (n-q)q^{n-1} + k(H(q,q)).
$
Let $n\ge 3$ and we assume that
$k(H(n-1,q)) \leq \alpha(n-1,q)$ holds for any $q$
such that $2 \leq q \leq n-1$.
Now consider a Hamming graph $H(n,q)$ where $q$ is an integer
satisfying $2 \leq q \leq n$.
If $n=q$, then the theorem clearly holds.
Suppose that $2\leq q \leq n-1$.
For $i \in [q]$,
let $H^{(i)}$ be the subgraph of $H(n,q)$ induced
by a vertex set
$\{ (x_1, x_2, \ldots, x_n) \in [q]^n \mid x_n=i\}$.
Then each $H^{(i)}$ is isomorphic to the Hamming graph $H(n-1,q)$.

We denote the minimum edge clique cover of $H^{(i)}$ by $\cF^{(i)}(n-1,q)$.
By the induction hypothesis,
it holds that $k(H^{(i)})= k(H(n-1,q)) \leq \alpha(n-1,q)$.
By Corollary \ref{cor:Dabc},
there exists an acyclic digraph $D^{(i)}$
for each $i\in [q]$ such that
\begin{itemize}
\item[(a)]  $C(D^{(i)})=H^{(i)} \cup I^{(i)}_{\alpha(n-1,q)}$,
\item[(b)]  $N^-_{D^{(i)}}(v) \in \cF^{(i)}(n-1,q) \cup \{ \emptyset \}$
for any $v \in V(D^{(i)})$,
\item[(c)] $|\{v \in V(D^{(i)}) \mid N^-_{D^{(i)}}(v) = \emptyset \}|
=\alpha(n-1,q) - (n-1-q)q^{n-2} = k(H(q,q))$.
\end{itemize}
where $I^{(i)}_{\alpha(n-1,q)}$ is a set of ${\alpha(n-1,q)}$ isolated vertices.
Then by (c), we may let, for each  $i\in [q]$,
\[
W^{(i)} :=\{v \in V(D^{(i)}) \mid N^-_{D^{(i)}}(v) = \emptyset \}
= \{w^{(i)}_1, \ldots, w^{(i)}_{k(H(q,q))} \}.
\]
Since $q \leq n-1$, it holds that
$k(H(q,q))=\alpha(n-1,q)-((n-1)-q)q^{n-2} \leq  \alpha(n-1,q)$.
By (a), there are at least $k(H(q,q))$ isolated vertices
in $I^{(i)}_{\alpha(n-1,q)}$ of $C(D^{(i)})$.
Let $J^{(i)}:= \{j^{(i)}_1, \ldots, j^{(i)}_{k(H(q,q))} \}$
be a set of $k(H(q,q))$ vertices which belong to $I^{(i)}_{\alpha(n-1,q)}$
in $C(D^{(i)})$. Let $D$ be the digraph defined by
\begin{eqnarray*}
V(D) &:=&
V(D^{(1)}) \cup
\bigcup_{i=2}^{q} (V(D^{(i)}) \setminus J^{(i)}), \\
A(D) &:=& A(D^{(1)}) \cup \bigcup_{i=2}^{q}
(A(D^{(i)}) \setminus
\{ (x,j^{(i)}_l) \mid 1 \leq l \leq k(H(q,q)),
x \in N_{D^{(i)}}^{-}(j^{(i)}_l) \} ) \\
&& \qquad \quad \,\,\cup
\bigcup_{i=2}^{q}
\{ (x,w^{(i-1)}_l) \mid 1 \leq l \leq k(H(q,q)),
x \in N_{D^{(i)}}^{-}(j^{(i)}_l) \}.
\end{eqnarray*}
Since each digraph $D^{(i)}$ is acyclic, the digraph $D$ is also acyclic.
In addition, $C(D)= (\bigcup_{i=1}^{q} H^{(i)}) \cup I_{k}$ where
$k := q \cdot \alpha(n-1,q) - (q-1) \cdot k(H(q,q))$.

Note that by the definition of $H^{(i)}$,
the graph $\bigcup_{i=1}^{q} H^{(i)}$ is a spanning subgraph of $H(n,q)$
except the edges of cliques in $\{S_n(x)\mid x \in [q]^{n-1} \}$.
Therefore, by letting$D^*$ be the digraph defined by
\begin{eqnarray*}
V(D^*) &:=& V(D) \cup I_{q^{n-1}}
\ = \ V(D) \cup \{z_{x} \mid x \in [q]^{n-1} \}, \\
A(D^*) &:=& A(D) \cup  \bigcup_{x\in  [q]^{n-1}} \{(y,z_x) \mid y\in S_n(x)\},
\end{eqnarray*}
we obtain an acyclic digraph  $D^*$ such that $C(D^*) = H(n,q) \cup I_{k^*}$,
where
\begin{eqnarray*}
k^* &:=& k + q^{n-1}\\
&= & q \cdot \alpha(n-1,q) - (q-1) \cdot k(H(q,q)) + q^{n-1}  \\
&=& q \cdot ((n-1-q)q^{n-2} + k(H(q,q)))
- (q-1) \cdot k(H(q,q)) +q^{n-1} \\
&=& (n-q)q^{n-1} + k(H(q,q))
\ = \ \alpha(n,q).
\end{eqnarray*}
Therefore, $k(H(n,q)) \leq \alpha(n,q)$.
Hence, the theorem holds.
\end{proof}

The following corollary follows from Theorem \ref{thm:UH_nq}.

\begin{Cor} \label{cor:UH_n3}
For $n\ge 3$, we have
$k(H(n,3)) \leq (n-3)3^{n-1}+6$.
\end{Cor}

In the following, we show a lower bound for the competition numbers
of ternary Hamming graphs $H(n,3)$ where $n\ge 3$.

\begin{Lem} \label{lem2}
Let $n \ge 3$, and let $G$ be a subgraph
of the ternary Hamming graph $H(n,3)$
with $10$ vertices.
Then the number of triangles in $G$
is at most $6$.
Moreover, it is exactly equal to $6$ if and only if $G$ is
isomorphic to either $H_1 \cup I_1$ or $H_2$ in Figure~\ref{fig:H1H2H3}.
\end{Lem}

\begin{proof}
We denote by $t_G$ the number of triangles in a graph $G$.
For a vertex $v$ in a graph $G$,
we denote by $t_G(v)$ the number of triangles in $G$
containing the vertex $v$.
For any graph $G$, it holds that
\begin{equation}\label{nG}
\sum_{v\in V(G)} t_G(v) = 3\times t_G.
\end{equation}
Let $G$ be a subgraph of $H(n,3)$ with $10$ vertices
such that $t_G$ is the maximum
among all the subgraphs of $H(n,3)$ with $10$ vertices.
Since the graph $H_2$ drawn in Figure
\ref{fig:H1H2H3}
is a subgraph of $H(n,3)$
with $10$ vertices and $6$ triangles,
we have $t_G \ge 6$.
Let $v_1, v_2, v_3, \ldots, v_9, v_{10}$
be the vertices of $G$.

\begin{figure}[t]
\begin{center}
\psfrag{a}{$H_1$}
\psfrag{b}{$H_2$}
\psfrag{c}{$H_3$}
\includegraphics{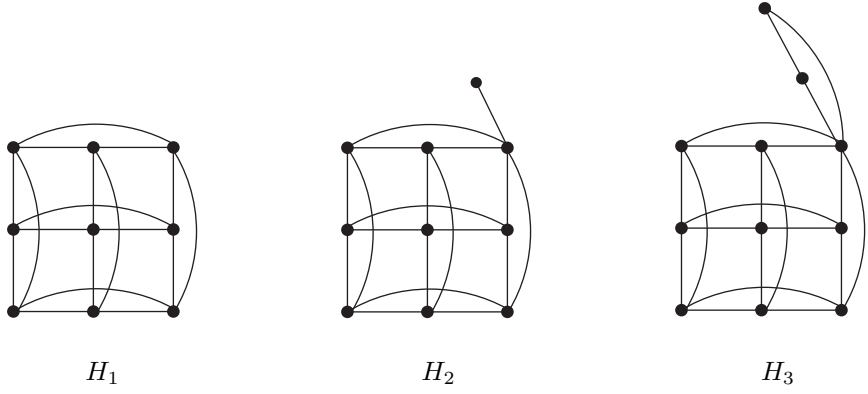}\\
\caption{Induced subgraphs of $H(n,3)$} \label{fig:H1H2H3}
\end{center}
\end{figure}

Suppose that there exists a vertex $v_i$
such that $t_G(v_i) \ge 3$. Then we will reach a contradiction.
Without loss of generality, we may assume that $t_G(v_1) \ge 3$
and that $\{v_1,v_2,v_3\}$, $\{v_1,v_4,v_5\}$, $\{v_1,v_6,v_7\}$
are cliques which contain $v_1$.
Since each edge is contained exactly one triangle by Lemma \ref{lem:LB2},
then the subgraph of $H(n,3)$ induced by $\{v_1, v_2, \ldots, v_7\}$
is isomorphic to
the graph $H_4$ drawn in Figure~\ref{fig;H4H5}.

\begin{figure}
\begin{center}
\psfrag{D}{$v_4$}
\psfrag{C}{$v_3$}
\psfrag{E}{$v_5$}
\psfrag{F}{$v_6$}
\psfrag{G}{$v_7$}
\psfrag{A}{$v_1$}
\psfrag{B}{$v_2$}
\psfrag{I}{$v_8$}
\psfrag{H}{$v_9$}
\psfrag{X}{$H_4$}
\psfrag{Y}{$H_5$}
\includegraphics{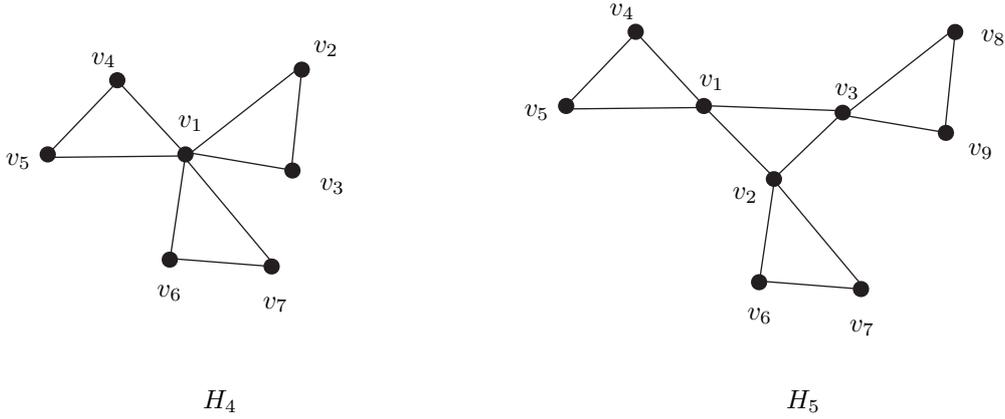} \\
\caption{Subgraphs of $H(n,3)$} \label{fig;H4H5}
\end{center}
\end{figure}

Since $G$ has at least $6$ triangles,
there exist at least three triangles $T_1,T_2,T_3$ in $G$
which are not in the induced subgraph $H_4$.
For each $1 \leq i\leq 3$,
the triangle $T_i$ has at least two vertices in $\{v_8,v_9,v_{10}\}$
since each edge of $G$ is contained in exactly one triangle.
We may assume that $T_1 \supseteq \{v_8,v_9\}$,
$T_2 \supseteq \{v_8,v_{10}\} $, and $T_3 \supseteq \{v_9,v_{10}\}$.
Then $\{v_8, v_9, v_{10}\}$ forms a triangle,
which contradicts the fact that each edge of $G$
is contained exactly one triangle.
Therefore, $t_G(v_i) \leq 2$ for all $v_i \in V(G)$.
Since $t_G(v_i) \leq 2$ for all $v_i \in V(G)$,
it holds that $\sum_{i=1}^{10} t_G(v_i) \le 20$.
By (\ref{nG}),
$\sum_{i=1}^{10} t_G(v_i)$ must be a multiple of $3$,
$\sum_{i=1}^{10} t_G(v_i) \le 18$.
Since $t_G \ge 6$, we also have
$\sum_{i=1}^{10} t_G(v_i) \ge 18$ by (\ref{nG}).
Thus, we have $\sum_{i=1}^{10} t_G(v_i) = 18$ and $t_G=6$.

Since $\sum_{i=1}^{10} t_G(v_i) = 18$ and $t_G(v_i) \leq 2$
for all $v_i \in V(G)$,
without loss of generality,
we may assume that $t_{G}(v_{10}) \le 1$.
Let $G - v_{10}$ be the graph obtained from $G$ by deleting $v_{10}$.
To show the ``moreover'' part,
it is sufficient to show that $G - v_{10}$ is isomorphic to the graph
$H_1 (\cong H(2,3))$ drawn in Figure \ref{fig:H1H2H3}.

Note that $15 \le \sum_{i=1}^{9} t_{G-v_{10}} (v_i)  \le 18$,
and $t_{G - v_{10}} (v) \le 2$ for all $v\in V(G - v_{10})$.
Then  $G - v_{10}$ has at least $6$ vertices $v$
such that $t_{G - v_{10}}(v) = 2$, and so
$G - v_{10}$ has a triangle $T_0=\{v_1,v_2,v_3\}$
such that $t_{G - v_{10}} (v_i)=2$ for any $i=1,2,3$.
For $i=1,2,3$, let $T_i$ be the triangle containing
$v_i$ and $T_i \neq T_0 $.
Then the triangles $T_1$, $T_2$, and $T_3$ are mutually vertex disjoint,
and we may assume that
$T_1=\{v_1, v_4, v_5\}$, $T_2=\{v_2, v_6, v_7\}$,
and $T_3=\{v_3, v_8, v_9\}$
(see the graph $H_5$ drawn
in Figure~\ref{fig;H4H5}).
Since $G - v_{10}$ has at least $6$ vertices $v$
such that $t_{G - v_{10}}(v) = 2$,
we may assume that $t_{G - v_{10}}(v_4) = 2$.
Then the triangle containing $v_4$ in $G - v_{10}$, which is not $T_1$,
are containing one of $\{v_6,v_7\}$ and one of $\{v_8,v_9\}$.
Without loss of generality, we may assume that $\{v_4,v_6,v_8\}$
is a triangle.

Since $\{v_1, v_2, v_6, v_4\}$ forms a cycle of length four,
without loss of generality, we may let
\[
\begin{array}{llll}
v_1 := (1,1,1, \ldots, 1), & v_2 := (2,1,1, \ldots, 1),
& v_4 := (1,2,1, \ldots, 1), & v_6 := (2,2,1, \ldots, 1).
\end{array}
\]
Since $\{v_1,v_2,v_3\}$ is a triangle,
we have $v_3=(3,1,1, \ldots,1)$.
Since $\{v_1,v_4,v_5\}$ is a triangle,
we have $v_5=(1,3,1, \ldots,1)$.
Since $\{v_2,v_6,v_7\}$ is a triangle,
we have $v_7=(2,3,1, \ldots,1)$.
Since $\{v_4,v_6,v_8\}$ is a triangle,
we have $v_8=(3,2,1, \ldots,1)$.
Since $\{v_3,v_8,v_9\}$ is a triangle,
we have $v_9=(3,3,1, \ldots,1)$.
Then $\{v_5,v_7,v_9\}$ forms a triangle,
and hence $G - v_{10}$ is isomorphic to $H_1$
in Figure \ref{fig:H1H2H3}.
We complete the proof.
\end{proof}

\begin{Thm} \label{thm:LH_n3}
For $n\ge 3$, we have
$k(H(n,3)) \geq (n-3)3^{n-1} +6$.
\end{Thm}

\begin{proof}
Suppose that $k(H(n,3)) \leq (n-3)3^{n-1}+5$.
Then, by Corollary \ref{cor:Dabc},
there exists an acyclic digraph $D$
such that $C(D)=H(n,3)\cup I_{(n-3)3^{n-1}+5}$,
$N^-_D(v) \in \cF(n,3) \cup \{\emptyset\}$ for any $v \in V(D)$,
and
$|\{ v \in V(D) \mid N^-_D(v) = \emptyset\}| =5$.
Let $v_1, v_2, \ldots, v_{(n-3)3^{n-1} + 5}$ be an acyclic ordering of $D$.
We may assume that the vertices $v_1, v_2, \ldots, v_5$
have no in-neighbors in $D$.
For $6 \le i \le 12$, let
$\cF_i:= \{ N^-_{D}(x) \mid x \in \{ v_1, v_2, \ldots, v_{i}, v_{i+1} \} \}$
and
let $G_i$ be the subgraph of $H(n,3)$ induced
by $\{ v_1,v_2,\ldots, v_{i} \}$.
Note that $\cF_i$ contains $i-4$ triangles
whose vertices are in $\{ v_1,v_2,\ldots, v_{i} \}$ and that
$G_i$ contains all the triangles in $\cF_i$.
Since $G_{10}$ is a subgraph of $H(n,3)$
with $10$ vertices containing $6$ triangles, by Lemma \ref{lem2},
$G_{10}$ is isomorphic to $H_1 \cup I_1$ or $H_2$,
where $H_1$ and $H_2$ are the graphs drawn in
Figure \ref{fig:H1H2H3}.
Since $G_{10}$ is a subgraph of $G_{11}$
and $G_{11}$ has $7$ triangles,
$G_{10}$ must be isomorphic to $H_2$
and $G_{11}$ must be isomorphic to $H_3$ in Figure \ref{fig:H1H2H3}.
Then we can observe that any subgraph of $H(n,3)$ with $12$ vertices
containing $H_3$ cannot have $8$ triangles.
However, $G_{12}$ is a subgraph of $H(n,3)$
with $12$ vertices and $8$ triangles,
which is a contradiction.
Hence $k(H(n,3)) \geq (n-3)3^{n-1}+6$.
\end{proof}

\begin{proof}[Proof of Theorem \ref{Mainthm:H(n,3)}]
Theorem \ref{Mainthm:H(n,3)} follows from
Corollary \ref{cor:UH_n3} and
Theorem \ref{thm:LH_n3}.
\end{proof}

\section{Concluding Remarks}

In this paper, we gave the exact values of
the competition numbers of ternary Hamming graphs.
Note that the bound given in Theorem \ref{thm:UH_nq} is tight
when $q=2$ or $3$, that is,
$k(H(n,q))=(n-q)q^{n-1} + k(H(q,q))$ holds for $n \geq q$ and $q \in\{2,3\}$.
We left a question for a further research
whether or not the bound in Theorem \ref{thm:UH_nq}
is tight for any $2 \le q \le n$.


\end{document}